\documentclass[12pt]{amsart}
\usepackage{amsmath, amsthm, amscd, amssymb}
\usepackage[small,nohug,head=littlevee]{diagrams}
\diagramstyle[labelstyle=\scriptstyle]

\newtheorem{thm}{Theorem}[section]
\newtheorem*{thm*}{Theorem}

\newtheorem{cor}[thm]{Corollary}
\newtheorem*{cor*}{Corollary}
\newtheorem{lem}[thm]{Lemma}

\newtheorem*{con*}{Conjecture}
\newtheorem*{prob*}{Problem}

\theoremstyle{definition}
\newtheorem{defn}[thm]{Definition}
\newtheorem{case}{Case}
\newtheorem*{ack}{Acknowledgements}

\theoremstyle{remark}

\newtheorem{example}{Example}

\newcommand{\Z}{\mathbb{Z}}
\newcommand{\R}{\mathbb{R}}

\begin{document}
\title{The filling problem in the cube}

\author[Dominic Dotterrer]{Dominic Dotterrer}
\address{Department of Mathematics, University of Toronto} \email{d.dotterrer@utoronto.ca}

\maketitle

\begin{abstract}
We prove an isoperimetric inequality for filling cellular cycles in a high dimensional cube with cellular chains.  In addition, we provide a family of cubical cellular cycles for which the exponent in the inequality is optimal.

\end{abstract}

\section{Introduction}

The modern isoperimetric problem was first formulated in \cite{FF60} and its solutions have had applications in various geometric contexts. These applications are not specific to measure-metric geometry; discrete and combinatorial geometries have also benefited from isoperimetric phenomena.

Isoperimetric estimates in combinatorial geometry are nothing new.  Many forms of combinatorial optimization can be naturally formulated as an isoperimetric-type problem.  In addition, these estimates have made their mark partly by supplying a broad range of (sometimes surprising) applications.

In the late 1960's, Kruskal and Katona explicitly solved (independently) an isoperimetric-type problem in the high-dimensional simplex, thereby obtaining the now famous Kruskal-Katona theorem.  This theorem provides necessary inequalities for the $f$-vector of any simplicial complex \cite{K68}, \cite{Kr69}.  Later, Lindstr\"om (\cite{Lind71}) leveraged the same isoperimetric-symmetrization technique to prove a similar result for all complexes which cellularly embed into a cube.

Harper (\cite{harper1964optimal}, with an addendum by Bernstein \cite{bernstein1967maximally}),  and Hart (\cite{hart1976note}) gave complete solutions to the edge-isoperimetric problem in the cube.  This isoperimetric estimate has spawned several interesting applications.  Harper used this estimate to obtain a sharp lower bound on the graph bandwidth of the cubical graph, while Hart analyzed the dynamics of some simple multi-player games.  In addition, Burtin (\cite{B77}) leveraged it to obtain a sharp phase transition in the connectivity of subgraphs of the cube.

Burtin may have been the first to {\it explicitly} recognize the relationship between the isoperimetric problem in graphs and the structure of random subgraphs.  Specifically, random subgraphs inherit their edge-isoperimetric behaviour (hence connectivity) from their ambient graph with high probability.
That observation manifested in higher dimensions early in this century when Linial and Meshulam (\cite{LM06}, \cite{MW09}) utilized a linear isoperimetric inequality in the simplex to obtain a sharp threshold on the vanishing of the cohomology of a family of random simplicial complexes.  These complexes are often considered to be the direct higher-dimensional analogue of Erd\H{o}s-R\'enyi random graphs (\cite{ER60}). However, this application of the isoperimetric inequality is in fact quite general, as observed in \cite{dotterrer2010coboundary}.

Recently, the applications of discrete isoperimetric inequalities has reached even farther.  A problem in combinatorial geometry dating back to the early 1980's (\cite{barany82}, \cite{boros1984number}) asks if it is possible to arrange $n$ points in $\R^k$ so that the convex hulls of subsets of $k+1$ points do not intersect too much.  Gromov greatly generalized this problem (\cite{gromov-sing}) and showed that linear (co-)isoperimetric inequalities are enough to obtain strong bounds (in fact, in the case of the simplex, the strongest known!) on such ``waists" of maps of cellular complexes to affine space. \cite{wagner2011gromov}  \cite{karasev2011simpler}

Along these same lines, Wagner (\cite{wagner2011minors}) showed that a linear co-isoperimetric inequality in a $k$-dimensional simplicial complex is an obstruction to its topological embeddability into $\R^{2k}$.  Also, Gromov and Guth (\cite{gromov2011generalizations}) have explored other notions of embedding complexity via isoperimetric inequalities.

\bigskip

This locus of ideas and applications have endorsed the re-emergence of combinatorial isoperimetric-type inequalities.  The purpose of this article is provide a combinatorial proof of an isoperimetric inequality for cellular cycles in the high-dimensional cube:

\begin{thm*}

There exists a constant $c_k$, which depends only on $k$, such that for every $k$-dimensional cellular $\Z_2$-cycle, $z \in Z_k Q_n$, in the $n$-dimesional cube, $Q_n$, there exists a chain $y \in C_{k+1} Q_n$, such that $\partial y = z$ and $$\lVert y \rVert \leq c_k \lVert z \rVert^{\frac{k+1}{k}}. $$

\end{thm*}

This proof is philosophically influenced by the proof given by Federer and Fleming (\cite{FF60}) of the linear spherical isoperimetric inequality and the proofs by Gromov (\cite{gromov1983filling}), and Wenger (\cite{wenger2008short}) of the Euclidean isoperimetric inequality.  However, our proof is not merely a discretization of those techniques.  In those settings scaling and symmetry play a very strong role, as well as the metric structure (i.e. diameter bounds).  The cube, for example, does not allow scaling methods and the symmetry is prescribed.  Our proof technique is therefore distinct while leveraging some of the intuition from Euclidean space. 

In addition, we will show that the above estimate is sharp up to a constant:

\begin{thm*}
There is a constant $\omega_k$, which depends only on $k$, and a family of $k$-dimensional $\Z_2$-cycles, $z^k_n \in Z_k Q_n$, such that for every $\Z_2$-chain, $y$, such that $\partial y = z$:

$$\lVert y \rVert \geq \omega_k \lVert z^k_n \rVert^{\frac{k+1}{k}}. $$

\end{thm*}

The paper is structured as follows:

We will briefly nail down our notation in the next section before proceeding to the proof of the isoperimetric inequality.  In the proof, we will have to start by proving a linear inequality, which will be used to obtain the dimension-free case.  Once we have our linear inequality, we will consider three basic cases for the geometry of the cycle to be filled.  We will handle each of these cases differently.

After the proof of the theorem, we will append two technical lemmas which, although essential to the proof, carry little real philosophical weight.

In the final section, we will provide the promised class of isoperimetric-minimizing cycles.

\begin{ack}
The author would like to thank Larry Guth for his suggestions, support and encouragement, Matt Kahle for providing the impetus to complete the project, and Boris Pittel for his enthusiastic questions and comments.  The author would also like to thank the Institute for Advanced Study, where the early stages of this work were completed.
\end{ack}

\section{Some notation}

We will denote the $n$-dimensional cube by $Q_n$.  We think of it as a cellular complex whose $k$-dimensional faces are binary strings of length $n$ with $k$ indeterminant entries.  For example,
$$( 0,0,1,*,0,1,*,*,1,1,*)$$
is a $4$-dimensional face of the $11$-dimensional cube.  We let $Q^{(k)}_n$ denote the set of $k$-dimensional faces, so that $|Q^{(k)}_n| = 2^{n-k} \binom{n}{k}$.

We define the set of $\Z_2$-cellular $k$-chains: $$C_k Q_n = \{ Q^{(k)}_n \to \Z_2 \}$$
and the boundary map: $$\cdots \rightarrow C_{k+1} Q_n \overset{\partial}{\rightarrow} C_{k}Q_n \overset{\partial}{\rightarrow}  C_{k-1} Q_n \rightarrow \cdots$$

Where the boundary of a single $k$-face is the sum of all the $(k-1)$-faces obtained by determining one of the $k$ indeterminant entries (as either $0$ or $1$).

We denote the set of $k$-cycles, $$Z_k Q_n = \ker \partial.$$

We endow each of the chain spaces $C_k Q_n \cong \Z^{Q^{(k)}_n}_2$ with the Hamming norm:
$$\lVert y \rVert = | {\rm supp} y |.$$

Now, we will explain a technique that we will use throughout.  Let us suppose that $z \in Z_k Q_n$.  Further, let us choose some $(n-1)$-dimensional face of $Q_n$ (by making all but one coordinate indeterminant).  We will denote this face $Q^+_{n-1}$ and its opposing face by $Q^-_{n-1}$.  There are three types of $k$-dimensional faces: those that lie in $Q^+$, those that lie in $Q^-$, and those that lie in neither.  As a result we can write 

$$z_- := z|_{Q^-} \quad z_+ := z|_{Q^+} \quad z_0 = z - z_+ - z_-.$$

We can think of these chains in the following way:$$z_+ , z_- \in C_k Q_{n-1} \quad z_0 \in Z_{k-1} Q_{n-1} \quad \text{     with      } \partial z_+ = z_0 =\partial z_-.$$

We will often call this choice of $(n-1)$-face (and the subsequent segregation of $z$) simply by $H$.

Finally, we will sometime use the following shorthand:

$${\rm Fill} (z) =\min \{ \lVert y \rVert \, | \, \partial y = z \} $$

\section{the filling problem in the cube}\label{main}

In this section, we will prove our main result.  There are several steps.  In order to obtain the dimension-free inequality, we will first need a {\it linear} isoperimetric inequality.  We will use it repeatedly on our way to obtaining the sharp exponent.

\subsection{The linear inequality}\label{linineq}

\begin{lem}\label{linear}  Let $z \in Z_k Q_n$ be a $k$-dimensional cellular $\Z_2$-cycle in the $n$-dimensional cube, $Q_n$.  There exists a chain $y \in C_{k+1} Q_n$ such that $\partial y = z$ and $$\lVert y \rVert \leq \frac{n-k}{2(k+1)} \lVert z \rVert.$$

\end{lem}

A proof of this lemma appears in \cite{gromov-sing}.  The proof is an adaptation of a technique developed by Federer and Fleming (\cite{FF60}) for proving a sharp linear isoperimetric inequality in the round sphere.

It is worth commenting on the content of this lemma.  Since the linear constant depends on $n$, the inequality provides a good filling of $z$ when $\lVert z \rVert \sim n^k$, but is not very useful when the cycle is significantly smaller than that.

\begin{proof}
We will induct on $n$.  When $n = k+1$, there is only one nontrivial cycle (it is the boundary of the only $(k+1)$-dimensional cell) and we see that $$\lVert y \rVert \leq \frac{1}{2(k+1)} \lVert z \rVert.$$

Now suppose we have proven the lemma for all cubes of dimension less than $n$.  Let us choose a $(n-1)$-dimensional face of $Q_n$ as described in the last section.

Now suppose we have a filling of $z_+ + z_- \in Z_k Q_{n-1}$ such that:
$$\lVert w \rVert \leq \frac{n-k-1}{2(k+1)} \lVert z_+ + z_- \rVert$$

Now we will think of $w$ as sitting in $Q^-_{n-1}$ and we will add some $(k+1)$-faces to it as follows.  For each $k$-face in $z_+$, include the {\it unique} $(k+1)$-face which contains it and is {\it not} contained in $Q^+_{n-1}$.  One easily checks that this new chain, $y$ has $\partial y = z$ and $$\lVert y \rVert \leq \lVert z_+ \rVert + \frac{n-k-1}{2(k+1)}  (\lVert z_+ \rVert + \lVert z_- \rVert)$$

Now if we average over all choice of slices, $H$, each $k$-face in $z$ is overcounted exactly $n-k$ times (because each $k$-face lies in $n-k$ many faces of dimension $n-1$).  So we have:

$$\mathbb{E} \lVert y \rVert = \frac{1}{2n} \sum \lVert y \rVert \leq \Bigl( \frac{n-k}{2n} \Bigr) \cdot \Bigl( 1+  \frac{n-k-1}{k+1} \Bigr) \lVert z \rVert = \frac{n-k}{2(k+1)} \lVert z \rVert$$

So therefore, there is some filling which satisfies the desired inequality.

\end{proof}

In the final section of this article, we will show that the constant obtained in this linear inequality is optimal.

\subsection{The dimension-free inequality}

\begin{thm}\label{main}

There exists a constant $c_k$, which depends only on $k$, such that for every $k$-dimensional cellular $\Z_2$-cycle, $z \in Z_k Q_n$, there exists a chain $y \in C_{k+1} Q_n$, such that $\partial y = z$ and $$\lVert y \rVert \leq c_k \lVert z \rVert^{\frac{k+1}{k}}. $$

\end{thm}

\begin{proof}

This proof is philosophically influenced by the proofs of the filling inequality given by M. Gromov (\cite{gromov1983filling}, \cite{systole-notes}) and S. Wenger (\cite{wenger2008short}).  The proof is by induction on $k$ and $n$.\\

{\bf Base Cases}.

For $n = k+1$, the inequality is trivial because there is only one nontrivial $k$-cycle.

For $k= 1$ and $n$ arbitrary, we observe the simple fact that if $\lVert z \rVert = 2m$ and $z$ is connected, then $z$ is contained entirely in some $m$-dimensional cube, $Q_m$.  As a result, we can apply the linear inequality:

$$\lVert y \rVert \leq \frac{m-1}{4} \lVert z \rVert \leq \frac{1}{8} \lVert z \rVert^2.$$

\bigskip

We will proceed, as in the case of the linear inequality, by chosing a separation, $H$, of the cube $Q_n$ into two opposite hyperfaces, $Q^+_{n-1}$ and $Q^-_{n-1}$.  Let us choose $\epsilon$ and $\delta$ to be some constants (to be explicitly determined later).  There are three cases to consider:\\

\begin{case} {\bf $[ \exists H, \; \lVert z_0 \rVert < \epsilon \lVert z \rVert^{\frac{k-1}{k}} \text{     and     } \lVert z_+ \rVert \leq \delta \lVert z_0 \rVert^{\frac{k}{k-1}} ]$} 

This condition states that we can find a slice, $H$,  which intersects a relatively small part of $z$ and either one side or the other (without loss of generality, $z_+$) is much smaller than a filling we would obtain from the isoperimetric inequality.

In this case, we will fill $z$ just as we did when proving the linear inequality

$$\begin{array}{rcl}
\lVert y \rVert &\leq &\lVert z_+ \rVert + c_k  \lVert z_- + z_+ \rVert^{\frac{k+1}{k}} \\
&\leq& \delta \lVert z_0 \rVert^{\frac{k}{k-1}} + c_k ( \lVert z \rVert - \lVert z_0 \rVert )^{\frac{k+1}{k}}
\end{array}$$  

Setting $L = \frac{\delta}{c_k}$ and ensuring that $\epsilon \leq \frac{(k+1)^k}{ (k+1)^k + (kL)^k }$ we can apply lemma \ref{tech2} (in the next section), we obtain that
$$\lVert y \rVert \leq c_k \lVert z \rVert^{\frac{k+1}{k}}. $$
\end{case}

\begin{case}{\bf $[ \exists H, \; \lVert z_0 \rVert < \epsilon \lVert z \rVert^{\frac{k-1}{k}} \text{     and both     } \lVert z_+ \rVert, \lVert z_- \rVert > \delta \lVert z_0 \rVert^{\frac{k}{k-1}} ]$} 

This condition states that we have found a slice in which we have cut $z$ into two large bulbs with a relatively small cut.  We will take advantage of this by filling $z_0$ using the isoperimetric inequality for $(k-1)$-cycles and then deal with the two resulting cycles in $Q^+_{n-1}$ and $Q^-_{n-1}$ separately.

Thinking of $z_0$ as a $(k-1)$-cycle, $z_0 \in Z_{k-1} Q_{n-1}$, we can find a filling of volume less than $c_{k-1} \lVert z_0 \rVert^{\frac{k}{k-1}}$.  This will leave us with cycles in $Q^+_{n-1}$ (resp. $Q^-_{n-1}$) of volume $\lVert z_+ \rVert + c_{k-1} \lVert z_0 \rVert^{\frac{k}{k-1}} $ (resp. $\lVert z_- \rVert$...).  Filling these cycles separately in each $(n-1)$-cube, we obtain:

$$\lVert y \rVert \leq c_{k-1} \lVert z_0 \rVert^{\frac{k}{k-1}} + c_k \Bigl[ ( \lVert z_+ \rVert + c_{k-1} \lVert z_0 \rVert^{\frac{k}{k-1}}  \rVert )^{\frac{k+1}{k}} + ( \lVert z_-\rVert + c_{k-1} \lVert z_0 \rVert^{\frac{k}{k-1}}  \rVert )^{\frac{k+1}{k}} \Bigr]  $$

Setting $$ x = \frac{ \lVert z_+ \rVert}{ \lVert z_+ \rVert + \lVert z_- \rVert} \text{      and        }  y = \frac{ \lVert z_- \rVert}{ \lVert z_+ \rVert + \lVert z_- \rVert} $$
we apply lemma \ref{tech1} and find that either
$$( \lVert z_+ \rVert + c_{k-1} \lVert z_0 \rVert^{\frac{k}{k-1}}  \rVert )^{\frac{k+1}{k}} + ( \lVert z_-\rVert + c_{k-1} \lVert z_0 \rVert^{\frac{k}{k-1}}  \rVert )^{\frac{k+1}{k}} \leq (\lVert z_+ \rVert + \lVert z_- \rVert)^{\frac{k+1}{k}}$$
or $$c_{k-1} \lVert z_0 \rVert^{\frac{k}{k-1}} \geq [2^{\frac{1}{k+1}} - 1] \min \{ \lVert z_+ \rVert, \lVert z_- \rVert \} \geq [2^{\frac{1}{k+1} }- 1]  \delta \lVert z_0 \rVert^{\frac{k}{k-1}} $$
Therefore, if we can choose $\delta > \frac{c_{k-1}}{2^{\frac{1}{k+1}} -1 }$ then we have,
$$\lVert y \rVert \leq  c_{k-1} \lVert z_0 \rVert^{\frac{k}{k-1}} + c_k (\lVert z_+ \rVert + \lVert z_- \rVert)^{\frac{k+1}{k}} $$

Now setting $L = \frac{c_{k-1}}{c_k}$ and again ensuring that $\epsilon \leq \frac{(k+1)^k}{ (k+1)^k + (kL)^k }$ we can apply lemma \ref{tech2} to obtain that
$$\lVert y \rVert \leq c_k \lVert z \rVert^{\frac{k+1}{k}}$$

\end{case}

\begin{case}{\bf $[ \forall H, \; \lVert z_0 \rVert \geq \epsilon \lVert z \rVert^{\frac{k-1}{k}}  ]$} 

This condition states that every slice is large.  As a result, we will see that the total volume of $z$ is large enough to simply apply the linear inequality.  More specifically, if we sum over all slices, each $k$-face in $z$ is sliced exactly $k$ times so:

$$\sum_H \lVert z_0 \rVert = k \lVert z \rVert  \geq n \epsilon \lVert z \rVert^{\frac{k-1}{k}} $$

$$\Longrightarrow \lVert z \rVert^{\frac{k+1}{k} } \geq \frac{n}{k} \epsilon \lVert z \rVert$$

so that if $\epsilon \geq \frac{1}{2c_k}$ we can simply apply the linear inequality:

$$\lVert y \rVert \leq \frac{n-k}{2(k+1)} \lVert z \rVert \leq  c_k \lVert z \rVert^{\frac{k+1}{k}}$$

\end{case}

\subsection{The constants}

Now that we have dealt with all three cases, let us wrap up our constants.  In all cases, we needed that

$$ \frac{ 1 }{  2c_k } \leq \epsilon \leq \frac{ (k+1)^k }{ (k+1)^k + (kL)^k }$$

Where $L$ either appeared as $\frac{c_{k-1}}{c_k}$ in the second case or as $\frac{\delta}{c_k} =\frac{c_{k-1}}{c_k} [2^{\frac{1}{k+1}}-1]^{-1}$ in the first (the value of $\delta$ was irrelevant in the first case).  Since, $\frac{\delta}{c_k} \geq \frac{c_{k-1}}{c_k}$, we can take $L = \frac{\delta}{c_k}$.

Now taking, say, $$c_k = \prod^k_{i=1} [2^{\frac{1}{i+1}} - 1]^{-1} = O(k!) $$
so that $L=1$, we see that it is possible to choose $\epsilon$ and $\delta$ consistently in each case.

\end{proof}

\section{Two small technical inequalities}

Here we prove the technical lemmas that we required in section \ref{main}.

\begin{lem}\label{tech1}

Suppose $x + y = 1$ and $(x + p)^{\frac{k+1}{k}} + (y + p)^{\frac{k+1}{k}} \geq 1$, then $$p \geq \Bigl( 2^{\frac{1}{k+1}} -1 \Bigr) \min \{ x, y\} $$

\end{lem}

This lemma appears in \cite{systole-notes} as lemma 8 (without the explicit constant).  For completeness, we provide it here as well.

\begin{proof}
Let $x \in [0, \frac{1}{2} ] $ and consider the locus of points $(x,p)$ satisfying:
$$(x + p)^{\frac{k+1}{k}} + (1-x+p)^{\frac{k+1}{k}} = 1.$$
We will first seek to show that $\frac{d^2 p}{dx^2} \leq 0$ when $x \in [0,\frac{1}{2}]$.
Implicitly differentiating:
$$ \frac{dp}{dx}  \Bigl( (1-x+ p)^{\frac{1}{k}} + (x+p)^{\frac{1}{k}} \Bigr)  = (1-x+ p)^{\frac{1}{k}} - (x+p)^{\frac{1}{k}} .$$

Noting that $$(1+ p') = \frac{2(1- x + d )^{\frac{1}{k}}}{ (1-x+ p)^{\frac{1}{k}} + (x+p)^{\frac{1}{k}} } \quad \text{ and } \quad ( p'-1) = \frac{-2( x + d )^{\frac{1}{k}}}{ (1-x+ p)^{\frac{1}{k}} + (x+p)^{\frac{1}{k}} }$$

and differentiating again:

$$p''  \Bigl( (1-x+ p)^{\frac{1}{k}} + (x+p)^{\frac{1}{k}} \Bigr)^2  +$$

$$ p' \Bigl[  (1-x+ p)^{\frac{1}{k}} (x+p)^{\frac{1-k}{k}}    - (x+p)^{\frac{1}{k}}    (1-x+ p)^{\frac{1-k}{k}}     \Bigr] =$$

$$ - \Bigl[  (1-x+ p)^{\frac{1}{k}} (x+p)^{\frac{1-k}{k}}    + (x+p)^{\frac{1}{k}}    (1-x+ p)^{\frac{1-k}{k}}     \Bigr] $$

And since on the interval $x \in [ 0, \frac{1}{2} ]$ we have $$(1-x+ p)^{\frac{1}{k}} (x+p)^{\frac{1-k}{k}}    \geq (x+p)^{\frac{1}{k}}    (1-x+ p)^{\frac{1-k}{k}}    $$

we see that $p'' \leq 0 $ on $x \in [0, \frac{1}{2} ]$.

Therefore $p \geq 2 \Bigl[  p(\frac{1}{2}) - p(0) \Bigr] x$.  Since $p(0) = 0$ and $p (\frac{1}{2} ) = \frac{1}{2} [ 2^{\frac{1}{k+1}} - 1 ]$, the lemma follows.

\end{proof}

\begin{lem}\label{tech2}

Let $e = \Bigl( \frac{ k+1}{k} \Bigr)^k$.  If $0 \leq x \leq \Bigl( \frac{eS}{e + L^k} \Bigr)^{\frac{k-1}{k}}$ then
$$(S-x)^{\frac{k+1}{k}} + Lx^{\frac{k}{k-1}} \leq S^{\frac{k+1}{k}}$$

\end{lem}

\begin{proof}

$$\begin{array}{rcl}
eS &\geq& ex^{\frac{k}{k-1}} + L^k x^{\frac{k}{k-1}} \\
\Longrightarrow e(S-x) &\geq& e( x^{\frac{k}{k-1}} - x ) + L^k x^{\frac{k}{k-1}} \geq L^k x^{\frac{k}{k-1}} \\
\Longrightarrow \frac{k+1}{k} (S-x)^{\frac{1}{k}} &\geq& Lx^{\frac{1}{k-1}} \\
\text{So since     } \; S^{\frac{k+1}{k}} - (S-x)^{\frac{k+1}{k}} &\geq& \Bigl[ \frac{d}{dw} |_{S-x} w^{\frac{k+1}{k}} \Bigr] x =  \frac{k+1}{k} (S-x)^{\frac{1}{k}}x \\
&\geq& Lx^{\frac{k}{k-1}} \\
\Longrightarrow S^{\frac{k+1}{k}} &\geq&  (S-x)^{\frac{k+1}{k}} + Lx^{\frac{k}{k-1}} 

\end{array}$$

\end{proof}

\section{Isoperimetric minimizers}

In this section, we will construct a family of cycles $z^k_n \in Z^kQ_n$, which will show that the exponent in theorem \ref{main} cannot be improved.

\begin{defn}
Recall, that a $k$-dimensional face in $Q_n$ is defined by allowing some $k$ coordinates to vary while fixing the other $n-k$ coordinates as either $0$ or $1$.
 
We will define $z^k_n \in Z^k Q_n$ as the set of faces satisfying the following:

\begin{enumerate}

\item allow {\it any} $k$ coordinates (say $i_1, \dots, i_k$) to vary.

\item for all coordinates, $i$, such that $i_j < i < i_{j+1}$, either $i=0$ or $i=1.$

\item if $i_j < i < i_{j+1}$ and $i_{j+1} < i' < i_{j+2}$, then $i = 0$ implies $i' = 1$ and $i=1$ implies $i' = 0$.

\end{enumerate}

\end{defn} 

This codification is probably not very clear at first glance, so let me give an example.

\begin{example}

The chain $z^3_{10}$ consists of all three faces of the form:

$$( 0,0,*, 1,*, 0, *, 1,1,1) \quad \text{   or say   } \quad (1,*,*,1,1,1,*,0,0,0)$$

This is simply to say that we specify $3$ (or $k$) coordinates to be allowed to vary and then in-between these coordinates, we force the entries to be constant in blocks, {\it alternating only once we reach a varying coordinate}.

\end{example}

\begin{lem} The chain $z^k_n$ satisfies the following:

\begin{enumerate}

\item $z^k_n$ is a cycle.

\item $\lVert z^k_n \rVert = \lVert z^{k-1}_{n-1} \rVert + \lVert z^k_{n-1} \rVert = 2\binom{n}{k}$.

\item $ {\rm Fill} (z^k_n)  = {\rm Fill} (z^{k-1}_{n-1} ) + {\rm Fill} (z^k_{n-1} )= \binom{n}{k+1}$.
\end{enumerate}
\end{lem}

It is important to note that in particular, $${\rm Fill} (z^k_n) = \frac{n-k}{2(k+1)} \lVert z^k_n \rVert$$ and so therefore, this class of cycles make the linear inequality (lemma \ref{linineq}) sharp.

\begin{proof}
The second statement is clear by definition.  However, we would like to point out a geometric way of seeing it.

Let us specify a splitting, $H$, of the cube, just as we have done above, by deleting the first entry of the binary string.  Then
$$(z^k_n)_+ + (z^k_n)_- = z^k_{n-1} \quad \text{    and    } \quad (z^k_n)_0 = z^{k-1}_{n-1}.$$

Now suppose we have a filling, $y$, of $z^k_n$.  Then $y_0$ is a filling of $z^{k-1}_{n-1}$ and $y_+ + y_-$ is a filling of $z^k_{n-1}$.  Therefore,
$$\lVert y \rVert = \lVert y_+ \rVert + \lVert y_0 \rVert + \lVert y_- \rVert \geq {\rm Fill}(z^{k-1}_{n-1}) + {\rm Fill}(z^k_{n-1}).$$ 

Now, ${\rm Fill}(z^k_{k+1}) = 1 = \binom{k+1}{k+1}$,  while ${\rm Fill} (z^0_n) = n = \binom{n}{1}$.  Therefore, by induction, $${\rm Fill}(z^k_n) = \binom{n}{k+1}.$$

Let us briefly remark on why $z^k_n$ is a cycle.  A $(k-1)$-face lies in the boundary of a face of $z^k_n$ if it consists of alternating blocks of $1$'s and $0$'s in between varying coordinates {\it except} in one spot, where the block alternates unprompted.  In that case, this $(k-1)$-face lies in exactly two faces of $z^k_n$: the one obtained by varying the $1$ (at the unprompted alternation), and the one obtained by varying the $0$.  To give an example, the $2$-face given by: $$(1,1,*,0,0,1,1,*,0,0)$$ belongs to both $$(1,1,*,0,*,1,1,*,0,0) \quad \text{   and   } \quad (1,1,*,0,0,*,1,*,0,0) \quad \text{  in  } z^3_{10}$$ 

\end{proof}

\begin{cor}
$${\rm Fill}(z^k_n) \geq \omega_k \lVert z^k_n \rVert^\frac{k+1}{k}$$

where $\omega_k \geq\frac{ \sqrt[k]{k!} }{2^\frac{k+1}{k} (k+1)} ( 1-\epsilon )$ for any $\epsilon$ and large enough $n$.
\end{cor}

In particular, this shows that the exponent in theorem \ref{main} is optimal (though possibly the constant can be improved; $c_k = O(k!)$ while $\omega_k = \Omega(1)$).

\bibliography{cube-filling.bib}

\begin{thebibliography}{10}

\bibitem{barany82}
I.~B\'ar\'any.
\newblock A generalization of carath\'eodory's theorem.
\newblock {\em Discrete Mathematics}, 40(2-3):141--152, 1982.

\bibitem{bernstein1967maximally}
A.J. Bernstein.
\newblock Maximally connected arrays on the n-cube.
\newblock {\em SIAM Journal on Applied Mathematics}, 15(6):1485--1489, 1967.

\bibitem{boros1984number}
E.~Boros and Z.~F{\"u}redi.
\newblock The number of triangles covering the center of an n-set.
\newblock {\em Geometriae Dedicata}, 17(1):69--77, 1984.

\bibitem{B77}
Yu.~D. Burtin.
\newblock On the probability of connectedness of a random subgraph of the
  n-cube.
\newblock {\em Problemy Peredachi Informatsii}, 13, 1977.

\bibitem{dotterrer2010coboundary}
D.~Dotterrer and M.~Kahle.
\newblock Coboundary expanders.
\newblock 2010.

\bibitem{ER60}
P.~Erd\H{o}s and A.~R{\'e}nyi.
\newblock {\em On the evolution of random graphs}.
\newblock Akad. Kiad{\'o}, 1960.

\bibitem{FF60}
H.~Federer and W.H. Fleming.
\newblock {Normal and integral currents}.
\newblock {\em Annals of Mathematics}, 72(3):458--520, 1960.

\bibitem{gromov1983filling}
M.~Gromov.
\newblock Filling riemannian manifolds.
\newblock {\em J. Differential Geom}, 18(1):1--147, 1983.

\bibitem{gromov-sing}
M.~Gromov.
\newblock {Singularities, Expanders and Topology of Maps. Part 2: from
  Combinatorics to Topology Via Algebraic Isoperimetry}.
\newblock {\em Geometric And Functional Analysis}, pages 1--111, 2010.

\bibitem{gromov2011generalizations}
M.~Gromov and L.~Guth.
\newblock Generalizations of the kolmogorov-barzdin embedding estimates.
\newblock {\em Arxiv preprint arXiv:1103.3423}, 2011.

\bibitem{systole-notes}
L.~Guth.
\newblock {Notes on Gromov's systolic estimate}.
\newblock {\em Geometriae Dedicata}, 123(1):113--129, 2006.

\bibitem{harper1964optimal}
L.H. Harper.
\newblock Optimal assignments of numbers to vertices.
\newblock {\em Journal of the Society for Industrial and Applied Mathematics},
  12(1):131--135, 1964.

\bibitem{hart1976note}
S.~Hart.
\newblock A note on the edges of the n-cube.
\newblock {\em Discrete Mathematics}, 14(2):157--163, 1976.

\bibitem{karasev2011simpler}
R.~Karasev.
\newblock A simpler proof of the
  boros--f{\"u}redi--b{\'a}r{\'a}ny--pach--gromov theorem.
\newblock {\em Discrete \& Computational Geometry}, pages 1--4, 2011.

\bibitem{K68}
G.O.H. Katona.
\newblock {A theorem of finite sets}.
\newblock {\em Theory of graphs}, pages 187--207, 1968.

\bibitem{Kr69}
JB~Kruskal.
\newblock {The number of s-dimensional faces in a complex: An analogy between
  the simplex and the cube}.
\newblock {\em Journal of Combinatorial Theory}, 6(1):86--89, 1969.

\bibitem{Lind71}
B.~Lindstr\"om.
\newblock The optimal number of faces in cubical complexes.
\newblock {\em Arkiv f\"or Matematik}, 8(3):245--257, 1971.

\bibitem{LM06}
N.~Linial and R.~Meshulam.
\newblock {Homological connectivity of random 2-complexes}.
\newblock {\em Combinatorica}, 26(4):475--487, 2006.

\bibitem{MW09}
R.~Meshulam and N.~Wallach.
\newblock {Homological connectivity of random k-dimensional complexes}.
\newblock {\em Random Structures \& Algorithms}, 34(3):408--417, 2009.

\bibitem{wagner2011minors}
U.~Wagner.
\newblock Minors in random and expanding hypergraphs.
\newblock In {\em Proceedings of the 27th annual ACM symposium on Computational
  geometry}, pages 351--360. ACM, 2011.

\bibitem{wagner2011gromov}
U.~Wagner et~al.
\newblock On gromov's method of selecting heavily covered points.
\newblock {\em Arxiv preprint arXiv:1102.3515}, 2011.

\bibitem{wenger2008short}
S.~Wenger et~al.
\newblock A short proof of gromov's filling inequality.
\newblock {\em Proceedings of the American Mathematical Society},
  136(8):2937--2942, 2008.

\end{thebibliography}
\bibliographystyle{plain}

\end{document}